\RequirePackage{fix-cm}
\documentclass[smallextended]{svjour3}       
\smartqed  
\usepackage{graphicx}
\usepackage{CJK}
\usepackage{booktabs}
\usepackage{epstopdf}
\usepackage{amsmath}
\usepackage{color}
\usepackage{amsfonts}
\usepackage{amssymb}
\usepackage{psfrag}
\usepackage{wrapfig}
\usepackage{url}
\usepackage{amscd}
\usepackage{multicol}
\usepackage{algorithm}
\usepackage{algorithmic}
\usepackage[square, comma, sort&compress, numbers]{natbib} 
\usepackage{subfig}
\usepackage{diagbox}
\usepackage{enumerate}
\usepackage{multirow}
\newcommand{\mR}{\mathbb{R}}
\newcommand{\mC}{\mathbb{C}}

\newcommand{\mH}{\mathbb{H}}

\newcommand{\mcH}{\mathcal{H}}
\newcommand{\mcS}{\mathcal{S}}
\newcommand{\mcC}{\mathcal{C}}

\newcommand{\bi}{\mbox{\boldmath{$i$}}}
\newcommand{\bj}{\mbox{\boldmath{$j$}}}
\newcommand{\bk}{\mbox{\boldmath{$k$}}}
\newcommand{\bq}{\mbox{\boldmath{$q$}}}
\newcommand{\bp}{\mbox{\boldmath{$p$}}}
\newcommand{\bx}{\mbox{\boldmath{$x$}}}
\newcommand{\by}{\mbox{\boldmath{$y$}}}
\newcommand{\bz}{\mbox{\boldmath{$z$}}}
\newcommand{\bt}{\mbox{\boldmath{$t$}}}
\newcommand{\bu}{\mbox{\boldmath{$u$}}}
\newcommand{\bv}{\mbox{\boldmath{$v$}}}
\newcommand{\bb}{\mbox{\boldmath{$b$}}}

\newcommand{\bU}{\mbox{\boldmath{$U$}}}

\newcommand{\tr}{\mbox{\rm tr}\, }
\newcommand{\rank}{\mbox{\rm rank}\, }
\newcommand{\re}{\mbox{\rm Re~}}
\newcommand{\im}{\mbox{\rm Im~}}

\makeatletter
\newenvironment{breakablealgorithm}
  {
   \begin{center}
     \refstepcounter{algorithm}
     \hrule height.8pt depth0pt \kern2pt
     \renewcommand{\caption}[2][\relax]{
       {\raggedright\textbf{\ALG@name~\thealgorithm} ##2\par}%
       \ifx\relax##1\relax 
         \addcontentsline{loa}{algorithm}{\protect\numberline{\thealgorithm}##2}%
       \else 
         \addcontentsline{loa}{algorithm}{\protect\numberline{\thealgorithm}##1}%
       \fi
       \kern2pt\hrule\kern2pt
     }
  }{
     \kern2pt\hrule\relax
   \end{center}
  }
\makeatother

\begin{document}

\title{Quaternion matrix decomposition and its theoretical implications \thanks{Research supported by NSFC Grants 11771269, NSFC Grants 11831002, GIFSUFE Grants CXJJ-2019-391, and Program for Innovative Research Team of Shanghai University of Finance and Economics. }
}


\author{Chang He         \and
        Bo Jiang         \and
        Xihua Zhu
}


\institute{Chang He \at
              Research Institute for Interdisciplinary Sciences, Shanghai University of Finance and Economics, Shanghai 200433, PR China. \\
              \email{changhe@163.shufe.edu.cn}           
           \and
           Bo Jiang \at
              Research Institute for Interdisciplinary Sciences, School of Information Management and Engineering, Shanghai University of Finance and Economics, Shanghai 200433, PR China.
              \email{isyebojiang@gmail.com}
           \and
           Xihua Zhu \at
              School of Information Management and Engineering, Shanghai University of Finance and Economics, Shanghai 200433, PR China. \\
              \email{zhuxihua@163.sufe.edu.cn}
}

\date{Received: date / Accepted: date}

\maketitle

\begin{abstract}
This paper proposes a novel matrix rank-one decomposition for quaternion Hermitian matrices, which admits a stronger property than the previous results in \cite{sturm2003cones,huang2007complex,ai2011new}. The enhanced property can be used to drive some improved results in joint numerical range,  $\mathcal{S}$-Procedure and quadratically constrained quadratic programming (QCQP) in the quaternion domain, demonstrating the capability of our new decomposition technique.

\end{abstract}

\section{Introduction}\label{intro}
In recent years, we have witnessed a burst of quaternion representations in many fields,
 including color imaging \cite{chen2015color,xu2015vector,chen2019low,miao2020low, chen2020low}, signal processing \cite{flamant2019time,flamant2018complete},  robotics \cite{chou1992quaternion}, rolling bearing fault diagnosis \cite{yi2017quaternion}, quaternion convolutional neural networks (QCNNs) \cite{zhu2018quaternion,parcollet2019quaternion}, etc. Moreover, there are some noticeable steps towards optimizing the corresponding quaternion represented problems. Specifically, Qi et al.\ \cite{qi2020quaternion, qi2021quaternion} conducted a systematic study on quaternion matrix optimization, and Flamant et al.\ \cite{flamant2021general} proposed a general framework for constrained convex quaternion optimization. In terms of algorithms in the quaternion domain, affine projection algorithms \cite{xu2015optimization} and learning algorithms \cite{jahanchahi2013class} based on gradient and Hessian have been proposed and analyzed. Hence, the increasing number of quaternion-represented applications and the studies on the associated optimization problems call for a deeper understanding of the quaternion structure that could lead to some efficient solution methods. In this paper, we shall focus on one algebraic quaternion structure: matrix rank-one decomposition, and show that such decomposition admits a stronger property than that in the real and complex domain by leveraging the intrinsic quaternion nature. We further show that such merit of rank-one decomposition can be extended to some of its theoretical implications such as $\mathcal{S}$-Procedure, joint numerical range and quadratically constrained quadratic programming (QCQP) in the quaternion domain, and improve the associated results in the real and complex domains.

The matrix rank-one decomposition that we discuss in this paper is a technique of decomposing a positive semidefinite Hermitian matrix into the sum of rank-one matrices to satisfy the so-called equal inner product property, i.e., the inner product between some given matrices and each rank-one term in the decomposition has the same value.
The first such type of decomposition was introduced by Sturm and Zhang in \cite{sturm2003cones} with the equal inner product property valid for  {\it one} matrix, and it was used as a key technique to establish the Linear Matrix Inequality presentation of a class of matrix cones with its quadratic form co-positive over the real domain. Moreover, such a decomposition technique was found to be useful in quadratic minimization \cite{ye2003new} and  designing approximation algorithms for biquadratic optimization \cite{ling2010biquadratic}.
Soon after the work of \cite{sturm2003cones}, Huang and Zhang \cite{huang2007complex} extended the
matrix rank-one decomposition to the complex domain such that the equal inner product property holds for {\it two} matrices. Interestingly, we find that such property remains valid for {\it four} matrices when the rank-one decomposition is conducted in the quaternion domain. Moreover, our proof is fully constructive and the corresponding computational procedure is summarized in
Algorithm \ref{ACQMD1}.

Theoretical implications of our novel matrix decomposition technique are quite versatile and yield stronger results than those in the real and complex domains. The first two theoretical implications are in joint numerical range and $\mathcal{S}$-Procedure, both of which have some fundamental impacts and wide applications in many fields.
In particular, the joint numerical range is an important tool in linear algebra and convex analysis, and it is found to be useful in spectrum analysis \cite{rasulov2019description} and quantum computing \cite{dirr2006new,rodman2016continuity}. $\mathcal{S}$-Procedure occupies a crucial position in the field of robust optimization \cite{anitescu2000degenerate,anitescu2002superlinearly,goldfarb2003robust}, statistics \cite{hoerl1970ridge}, signal processing \cite{luo2003applications}, among others \cite{I2007A}. With our matrix rank-one decomposition result, we manage to establish the convexity of joint numerical range \cite{au1979remark,pang2004joint} for {\it five} matrices and the {\it lossless} of $\mathcal{S}$-Procedure for {\it four} Hermitian forms. As a comparison, similar results only hold for fewer matrices in the real and complex domains. In addition, our {result} can also be applied to quadratically constrained quadratic programming (QCQP). To be specific, when the number of quadratic constraints is no larger than $4$, we show that a rank-one solution of the SDP relaxation of (QCQP), which is hence an optimal solution of (QCQP), can be recovered from our matrix rank-one decomposition technique.

This paper is organized as follows. In Section \ref{sec:1}, we introduce some notations and definitions used throughout this paper. Section \ref{sec:qua_matrix_rank_one} is devoted to the new quaternion matrix rank-one decomposition theorem. To showcase the capability of our new theorem, we illustrate some improved results in the joint numerical range in Section \ref{sec:joint_num_range} and the $\mathcal{S}$-Procedure in Section \ref{sec:s_lem}. Finally, we present how to solve the (QCQP) with our novel decomposition technique as another theoretical implication of our result in Section \ref{sec:qua_quadratic_opt}.

\section{Preliminaries}\label{sec:1}

In this section, we introduce some basic algebraic operations in the quaternion domain for scalars, vectors and matrices.

\subsection{Quaternion Operations for Scalars}\label{sec:qua_alg}

We define the set of quaternions $\mH$ as a 4-dimensional normed division algebra over the real numbers $\mR$.
It has a canonical basis $\{1, \bi, \bj, \bk\}$, where $\bi, \bj, \bk$ are imaginary units such that
\begin{equation}\label{ima_multi}
\bi^2=\bj^2=\bk^2=\bi\bj\bk=-1,\quad \bi\bj=-\bj\bi=\bk.
\end{equation}
Then, any quaternion $q \in \mH$ can be written as
$$q = q_a + q_b\bi + q_c\bj + q_d\bk,$$
where $q_a,q_b,q_c,q_d \in \mR$ are the components of $q$. The real and imaginary parts of $q$ are denoted as $\re q=q_a$ and $\im q=q_b\bi + q_c\bj + q_d\bk$ respectively.
Note that, in contrast with the product operation in the real and complex domain, the product operation in the quaternion domain is noncommutative, i.e., $q \cdot p \ne p \cdot q$ for $p, q \in \mH$.

We denote by $\overline{q}=\re q-\im q$ the quaternion conjugate of $q$, and it holds that $\overline{(p \cdot q)}=\overline{q} \cdot \overline{p}$. For a given quaternion $q \in \mH$, $|q|$ denotes its modulus and can be expressed as $$|q|=\sqrt{q \cdot \overline{q}}=\sqrt{\overline{q} \cdot q}=\sqrt{q_a^2+q_b^2+q_c^2+q_d^2}.$$
Then, we denote $q = |q|(\cos\theta + q^I\sin\theta)$ as the triangle representation \cite{fan2012qua} of quaternion $q$, where $\cos\theta =  \frac{q_a}{|q|}$, $\sin\theta = \frac{\sqrt{q_b^2 + q_c^2 + q_d^2}}{|q|}$ and $q^I = \frac{q_b\bi + q_c\bj + q_d\bk}{\sqrt{q_b^2 + q_c^2 + q_d^2}}$.
At last, any non-zero quaternion $q$ has an inverse $q^{-1}=\overline{q}/|q|^2$ and the inverse of the product of two quaternions is $(p \cdot q)^{-1}=q^{-1} \cdot p^{-1}$.

\subsection{Quaternion Vectors and Quaternion Matrices}
Similar to the scalar case, a quaternion vector $\bq \in \mH^n$ can be written as
$$\bq = \bq_a + \bq_b\bi + \bq_c\bj + \bq_d\bk,$$
where $\bq_a,\bq_b,\bq_c,\bq_d \in \mR^n$ are the components of $\bq$. For a quaternion vector $\bq$, $\bq^{\top}$ denotes the transpose of $\bq$, and $\bq^{H}=(\overline{\bq})^{\top} = \overline{(\bq^{\top})}$ denotes its conjugate transpose. Any quaternion matrix $A\in\mH^{m\times n}$ can be expressed as
$$A = A_a+A_b\bi+A_c\bj+A_d\bk$$
with $A_a,A_b,A_c,A_d \in \mR^{m\times n}$. The transpose and the conjugate transpose of $A$ are $A^{\top}$ and $A^H=(\overline{A})^{\top} = \overline{(A^{\top})}$, respectively. For two quaternion matrices $A \in \mH^{m \times n}$ and $B \in \mH^{n \times p}$,  we have $(AB)^H = B^HA^H$; however, due to the noncommutativity of the product operation, $(AB)^\top \not= B^\top A^\top$ and $\overline{(AB)} \not= \overline{A} \ \overline{B}$.
For two vectors $\bq, \bp \in \mH^n$, their inner product
$$\bq \bullet \bp = \re(\bq^H\bp) = \bq_a^\top\bp_a+\bq_b^\top\bp_b+\bq_c^\top\bp_c+\bq_d^\top\bp_d.$$
Similarly, for two matrices $A ,B \in \mH^{n \times n}$, their inner product is defined as
$$A \bullet B := \re(\tr A^HB)=\tr(A_a^\top B_a + A_b^\top B_b + A_c^\top B_c + A_d^\top B_d),$$
where `tr' denotes the trace of a matrix.

\subsection{Hermitian and Positive Semidefinite Matrices}
We call $X$ a quaternion Hermitian matrix if it satisfies $X = X^H$. The cone of quaternion Hermitian matrices is denoted as $\mcH^{n}$. Then $\bu^HX\bu$ is real for all $\bu \in \mH^n$ if and only if $X \in \mcH^{n}$.
We denote by $\mcS_{+}^{n}(\mcS_{++}^{n})$, $\mcC_{+}^{n}(\mcC_{++}^{n})$ and $\mcH_{+}^{n}(\mcH_{++}^{n})$ the cones of real symmetric positive semidefinite (positive definite), complex Hermitian positive semidefinite (positive definite) and quaternion Hermitian positive semidefinite (positive definite) matrices, respectively. The notation $X \succeq 0$ ($X \succ 0$) means that $X$ is positive semidefinite (positive definite). Then, for any matrix $X \succeq 0$ ($X \succ 0$), we have $\bu^HX\bu$ is real and nonnegative (positive) for all $\bu \in \mH^n$ ($0 \neq \bu \in \mH^n$).

\section{The Quaternion Matrix Rank-One Decomposition Method}\label{sec:qua_matrix_rank_one}
In this section, we discuss a particular rank-one decomposition of matrices over the quaternion domain such that the inner product between some given matrices and each rank-one term in the decomposition has the same value.


In the real domain, Sturm and Zhang \cite{sturm2003cones} showed that for a rank-$r$ symmetric positive semidefinite matrix $Y \in \mcS_+^{n}$ and a real symmetric matrix $B \in \mcS^{n}$, there is a rank-one decomposition of $Y$ such that:
$$
Y = \sum_{i=1}^r\by_i\by_i^\top~~\mbox{\rm and}~~\by_j^\top B\by_j= (\by_i\by_i^\top) \bullet Y = \frac{B \bullet Y}{r},~~\mbox{for}~j=1,2,\cdots,r.
$$
Subsequently, Huang and Zhang \cite{huang2007complex} proved that the above decomposition could hold for two matrices in the complex domain. In particular, suppose $Z\in\mathcal{C}_+^n$ is a complex Hermitian positive semidefinite matrix of rank $r$, and $C_1,C_2 \in \mathcal{C}^n$ are two given complex Hermitian matrices. Then, there is a rank-one decomposition of $Z$ such that:
$$
Z = \sum_{i=1}^r\bz_i\bz_i^H~~\mbox{\rm and}~~\bz_j^H C_k\bz_j = \frac{C_k \bullet Z}{r},  ~~\mbox{for}~j=1,2,\cdots,r~\mbox{\rm and}~k=1,2.
$$
Thus, a natural question arises: can a similar decomposition holds for more matrices over the quaternion domain? The following theorem is the main result of this paper and gives an affirmative answer to this question.


\begin{theorem}\label{qua_decomposition}
    Let $A_k \in \mcH^{n}$, $k = 1, 2, 3, 4$, and rank-$r$ matrix $X \succeq 0$. There exists a rank-one decomposition of $X$ such that
    \begin{equation}\label{decom_X}
        X = \sum_{i=1}^r\bx_i\bx_i^H~~\mbox{\rm and}~~\bx_i^HA_k\bx_i = \frac{A_k \bullet X}{r},~~\mbox{for}~~i = 1,\cdots,r
    \end{equation}
  and $k = 1, 2, 3, 4$.
\end{theorem}
\begin{proof}
	We start with a weaker version of \eqref{decom_X} such that it holds only for two matrices $A_1$ and $A_2$, i.e., there exists a rank-one decomposition of $X$ such that:
	   \begin{equation}\label{decom_X_12}X = \sum_{i=1}^r\bu_i\bu_i^H ~~\mbox{\rm and}~~ \bu_i^HA_k\bu_i = \frac{A_k \bullet X}{r}, ~\mbox{\rm for}~ i = 1,\cdots,r~\mbox{and}~ k = 1, 2
	   \end{equation}
	according to the similar argument of Theorem 2.1 in \cite{huang2007complex}. The rest of the proof consists of two steps.\\
	{\bf Step $1$: prove \eqref{decom_X} for $k=1,2,3$}
	
	The conclusion follows
    if the decomposition in \eqref{decom_X_12} is also valid for matrix $A_3$, i.e., $\bu_i^HA_3\bu_i = \frac{A_3 \bullet X}{r}$, $i = 1,\cdots,r$. Otherwise, without loss of generality, there exist two vectors $\bu_1$ and $\bu_2$ such that
    $$\bu_1^HA_3\bu_1 > \frac{A_3 \bullet X}{r} \ \mbox{\rm and} \ \bu_2^HA_3\bu_2 < \frac{A_3 \bullet X}{r}.$$
    Denote $\bu_1^HA_1\bu_2 = a_1 + b_1\bi + c_1\bj + d_1\bk$ and $\bu_1^HA_2\bu_2 = a_2 + b_2\bi + c_2\bj + d_2\bk$. Let $\omega = \omega_a + \omega_b\bi + \omega_c\bj + \omega_d\bk \in \mH$ and construct
    $$\bv_1 = \frac{\bu_1\omega + \bu_2}{\sqrt{1 + |\omega|^2}} \ , \ \bv_2 = \frac{-\bu_1 + \bu_2\overline{\omega}}{\sqrt{1 + |\omega|^2}}.$$
    Then, it is easy to verify that
    \begin{align}\label{v-equal-u}
        &\bv_1\bv_1^H + \bv_2\bv_2^H \notag \\
        = &\frac{1}{1 + |\omega|^2}\left((\bu_1\omega + \bu_2)(\bu_1\omega + \bu_2)^H\right) + \frac{1}{1 + |\omega|^2}\left((-\bu_1 + \bu_2\overline{\omega})(-\bu_1 + \bu_2\overline{\omega})^H\right) \notag \\
        = &\frac{1}{1 + |\omega|^2}\big{(}(\bu_1\omega\overline{\omega}\bu_1^H + \bu_1\omega \bu_2^H + \bu_2\overline{\omega}\bu_1^H + \bu_2\bu_2^H) \notag \\
          &\qquad\qquad + (\bu_1\bu_1^H - \bu_1\omega \bu_2^H - \bu_2\overline{\omega}\bu_1^H + \bu_2\overline{\omega}\omega \bu_2^H)\big{)} \notag \\
        = &\frac{|\omega|^2\bu_1\bu_1^H + \bu_2\bu_2^H + \bu_1\bu_1^H + |\omega|^2\bu_2\bu_2^H}{1 + |\omega|^2} \notag \\
        = &\bu_1\bu_1^H + \bu_2\bu_2^H.
    \end{align}
    Note that the identity above holds for any $\omega$, which will allow us some flexibility to find an appropriate $\omega$ such that
    \begin{equation}\label{vAv_2}
  (1 + |\omega|^2)\bv_1^HA_k\bv_1=(1 + |\omega|^2)\frac{A_k \bullet X}{r},~~\mbox{for}~~k=1,2.
    \end{equation}
 In particular, observe that
    \begin{align}\label{vAv_i}
        &(1 + |\omega|^2)\bv_1^HA_k\bv_1 \\ \notag
        = &(\overline{\omega}\bu_1^H + \bu_2^H)A_k(\bu_1\omega + \bu_2) \\ \notag
        = &\overline{\omega}\bu_1^HA_k\bu_1\omega + \bu_2^HA_k\bu_2 + \overline{\omega}\bu_1^HA_k\bu_2 + \bu_2^HA_k\bu_1\omega \\ \notag
        = &(1 + |\omega|^2)\frac{A_k \bullet X}{r} + 2\re(\overline{\omega}\bu_1^HA_k\bu_2) \\ \notag
        = &(1 + |\omega|^2)\frac{A_k \bullet X}{r}, ~\mbox{for }~k=1,2 \notag
    \end{align}
    as long as
    \begin{align}
        &\re(\overline{\omega}\bu_1^HA_k\bu_2) \\ \notag
        = &\re\left((\omega_a - \omega_b\bi - \omega_c\bj - \omega_d\bk)(a_k + b_k\bi + c_k\bj + d_k\bk)\right) \\ \notag
        = &a_k\omega_a + b_k\omega_b + c_k\omega_c + d_k\omega_d=0. \notag
    \end{align}
Therefore, any solution of the following {equations}
    \begin{equation}\label{qua_sys_abcd}
\begin{cases}
a_1\omega_a + b_1\omega_b + c_1\omega_c + d_1\omega_d = 0 \\
a_2\omega_a + b_2\omega_b + c_2\omega_c + d_2\omega_d = 0,
\end{cases}
\end{equation}
will make \eqref{vAv_2} valid. Moreover, for
a particular solution $(\omega_{a*}, \omega_{b*}, \omega_{c*}, \omega_{d*})$ of \eqref{qua_sys_abcd} with $\omega_{a*}^2 + \omega_{b*}^2 + \omega_{c*}^2+ \omega_{d*}^2 =1$, $\omega(\alpha) = \alpha(\omega_{a*}+ \omega_{b*}\bi + \omega_{c*}\bj + \omega_{d*}\bk )$ is also a solution of \eqref{qua_sys_abcd} with $ |\omega(\alpha)|=\alpha$ for any $\alpha > 0$, as \eqref{qua_sys_abcd} is homogeneous. To extend the validness of \eqref{vAv_2} to matrix $A_3$ for some $\alpha$, i.e.,
 \begin{equation}\label{vAv_3_opt}
(1 + \alpha^2)\bv_1^HA_3\bv_1= (1 + |\omega(\alpha)|^2)\bv_1^HA_3\bv_1=(1 + \alpha^2)\frac{A_3 \bullet X}{r},
\end{equation}
we compute
\begin{align}\label{vAv_3}
&(1 + |\omega(\alpha)|^2)\bv_1^HA_3\bv_1 \\ \notag
= &(\overline{\omega(\alpha)}\bu_1^H + \bu_2^H)A_3(\bu_1\omega(\alpha) + \bu_2) \\ \notag
= &\overline{\omega(\alpha)}\bu_1^HA_3\bu_1\omega(\alpha) + \bu_2^HA_3\bu_2 + 2\re(\overline{\omega(\alpha)}\bu_1^HA_3\bu_2) \\ \notag
= &|\omega(\alpha)|^2\bu_1^HA_3\bu_1 + \bu_2^HA_3\bu_2 + 2\re(\overline{\omega(\alpha)}\bu_1^HA_3\bu_2)\\
= &\alpha^2\bu_1^HA_3\bu_1 + \bu_2^HA_3\bu_2 + 2\re(\overline{\omega(\alpha)}\bu_1^HA_3\bu_2). \notag
\end{align}
To proceed, we leverage the triangular representations we provided above to
represent $\re(\overline{\omega(\alpha)}\bu_1^HA_3\bu_2)$.
In particular, denoting $\tau_3:=\bu_1^HA_3\bu_2 = a_3 + b_3\bi + c_3\bj + d_3\bk$, we have $\omega(\alpha) = |\omega(\alpha)|(\cos\theta_\alpha + {\omega^I(\alpha)}\sin\theta_\alpha)$ and $\bu_1^HA_3\bu_2 = |\tau_3|(\cos\theta_3 + \tau_3^I\sin\theta_3)$.
 Then
    \begin{align}\label{triangu-v3}
        &\re(\overline{\omega(\alpha)}\bu_1^HA_3\bu_2) \\ \notag
        = &|\omega(\alpha)||\tau_3|\re[(\cos\theta_\alpha - {\omega^I(\alpha)}\sin\theta_\alpha)(\cos\theta_3 + {\tau^I_3}\sin\theta_3)] \\ \notag
        = &\alpha|\tau_3|[\cos\theta_\alpha\cos\theta_3 - \sin\theta_\alpha\sin\theta_3 \re({\omega^I(\alpha)} \cdot {\tau^I_3})], \notag
    \end{align}
    where $\cos\theta_\alpha = \frac{\alpha \omega_{a*}}{\alpha \sqrt{\omega_{a*}^2 + \omega_{b*}^2 + \omega_{c*}^2 + \omega_{d*}^2}} = \omega_{a*}$, $\sin\theta_\alpha = \sqrt{ \omega_{b*}^2 + \omega_{c*}^2 + \omega_{d*}^2}$ and ${\omega^I(\alpha)} = \frac{\omega_{b*}\bi + \omega_{c*}\bj + \omega_{d*}\bk}{\sqrt{\omega_{b*}^2 + \omega_{c*}^2 + \omega_{d*}^2}}$ are all independent of $\alpha$.
    Combing \eqref{vAv_3} and \eqref{triangu-v3}, identity \eqref{vAv_3_opt} that we want to prove is equivalent to
    \begin{equation}\label{root_function}
        \left(\bu_1^HA_3\bu_1 - \frac{A_3 \bullet X}{r}\right)\alpha^2 + 2\iota \cdot \alpha  + \left(\bu_2^HA_3\bu_2 - \frac{A_3 \bullet X}{r}\right) = 0,
    \end{equation}
    where $\iota = |\tau_3|\left(\cos\theta_\alpha\cos\theta_3 - \sin\theta_\alpha\sin\theta_3 \re({\omega^I(\alpha)} \cdot {\tau^I_3})\right)$. Then \eqref{root_function} is just a quadratic equation in $\alpha$, and it must have two real roots with opposite signs since $\bu_1^HA_3\bu_1 > \frac{A_3 \bullet X}{r}$ and $\bu_2^HA_3\bu_2 < \frac{A_3 \bullet X}{r}$. Let $\alpha_*$ be the positive root of \eqref{root_function} and construct $\omega_* = \alpha_*(\omega_{a*} + \omega_{b*}\bi + \omega_{c*}\bj + \omega_{d*}\bk)$. Then \eqref{vAv_3_opt} holds with $\alpha = \alpha_*$, and thus $\bv_1^HA_k\bv_1 = \frac{A_k \bullet X}{r}$ for $k = 1, 2, 3$. Since $X = \sum_{i=1}^r\bx_i\bx_i^H$, we have
    $$ X - \bv_1\bv_1^H \overset{\eqref{v-equal-u}}= \sum_{j = 3}^r\bu_j\bu_j^H + \bv_2\bv_2^H \succeq 0. $$
    Therefore, rank$(X - \bv_1\bv_1^H) = r -1$, $ \bu_i^HA_k\bu_i \overset{\eqref{decom_X_12}} = \frac{A_k \bullet X}{r}$, $i = 3,\cdots,r$ for $k = 1, 2$, and \begin{eqnarray*}\bv_2^HA_k\bv_2 &=&A_k \bullet (\bv_2 \bv_2^H) = A_k \bullet \left(X - \sum_{j = 3}^r\bu_j\bu_j^H - \bv_1\bv_1^H\right) \\
    	& \overset{\eqref{vAv_2}\,\eqref{decom_X_12}}=&A_k \bullet X - (r-2) \frac{A_k \bullet X }{r} - \frac{A_k \bullet X }{r}  = \frac{A_k \bullet X }{r},~~\forall ~~k=1,2.
    	\end{eqnarray*}
    Then, we can recursively repeat the above process on $X - \bv_1\bv_1^H$ and finally obtain a rank-one decomposition of $X$ such that
    \begin{equation}\label{decom_X_123}
        X = \sum_{i=1}^r\bv_i\bv_i^H~~\mbox{\rm and}~~\bv_i^HA_k\bv_i = \frac{A_k \bullet X}{r},~\mbox{for}~i = 1,\cdots,r~\mbox{and}~k=1,2,3.
    \end{equation}
    {\bf Step $2$: prove \eqref{decom_X} for $k=1,2,3,4$}

The conclusion follows
if the decomposition in \eqref{decom_X_123} is also valid for matrix $A_4$, i.e., $\bv_i^HA_4\bv_i = \frac{A_4 \bullet X}{r}$, $i = 1,\cdots,r$. Otherwise, without loss of generality, there exist two vectors $\bv_1$ and $\bv_2$ such that
    $$\bv_1^HA_4\bv_1 > \frac{A_4 \bullet X}{r} \ \mbox{\rm and} \ \bv_2^HA_4\bv_2 < \frac{A_4 \bullet X}{r}.$$
    Denote $\bv_1^HA_1\bv_2 = \hat{a}_1 + \hat{b}_1\bi + \hat{c}_1\bj + \hat{d}_1\bk$, $\bv_1^HA_2\bv_2 = \hat{a}_2 + \hat{b}_2\bi + \hat{c}_2\bj + \hat{d}_2\bk$ and $\bv_1^HA_3\bv_2 = \hat{a}_3 + \hat{b}_3\bi + \hat{c}_3\bj + \hat{d}_3\bk$. Let $\hat{\omega} = \hat{\omega}_a + \hat{\omega}_b\bi + \hat{\omega}_c\bj + \hat{\omega}_d\bk \in \mH$ and construct
    $$\bx_1 = \frac{\bv_1\hat{\omega} + \bv_2}{\sqrt{1 + |\hat{\omega}|^2}} \ , \ \bx_2 = \frac{-\bv_1 + \bv_2\overline{\hat{\omega}}}{\sqrt{1 + |\hat{\omega}|^2}}.$$
    Then, similar to \eqref{v-equal-u}, we can verify that
    \begin{equation}\label{x-equal-v}
    \bx_1\bx_1^H + \bx_2\bx_2^H=\bv_1\bv_1^H + \bv_2\bv_2^H
    \end{equation}
    holds for any $\hat{\omega}$. Moreover, when
    $\re(\overline{\hat{\omega}}\bv_1^HA_k\bv_2) = 0$, or equivalently $\hat{a}_k\hat{\omega}_a + \hat{b}_k\hat{\omega}_b + \hat{c}_k\hat{\omega}_c + \hat{d}_k\hat{\omega}_d = 0$, it holds that
      \begin{equation}\label{vAv_k}
  (1 + |\hat{\omega}|^2)\bx_1^HA_k\bx_1=(1 + |\hat{\omega}|^2)\frac{A_k \bullet X}{r},~~\mbox{for}~~k=1,2,3.
  \end{equation}
    In other words, any solution of the following {equations}
    \begin{equation}\label{qua_sys_abcd3}
        \begin{cases}
            \hat{a}_1\hat{\omega}_a + \hat{b}_1\hat{\omega}_b + \hat{c}_1\hat{\omega}_c + \hat{d}_1\hat{\omega}_d = 0 \\
            \hat{a}_2\hat{\omega}_a + \hat{b}_2\hat{\omega}_b + \hat{c}_2\hat{\omega}_c + \hat{d}_2\hat{\omega}_d = 0 \\
            \hat{a}_3\hat{\omega}_a + \hat{b}_3\hat{\omega}_b + \hat{c}_3\hat{\omega}_c + \hat{d}_3\hat{\omega}_d = 0,
        \end{cases}
    \end{equation}
    will make \eqref{vAv_k} valid. Since \eqref{vAv_k} is homogeneous and has three identities and four variables, we can find
    a particular solution $(\hat{\omega}_{a*}, \hat{\omega}_{b*}, \hat{\omega}_{c*}, \hat{\omega}_{d*})$ of \eqref{qua_sys_abcd3} with $\hat{\omega}_{a*}^2 + \hat{\omega}_{b*}^2 + \hat{\omega}_{c*}^2+ \hat{\omega}_{d*}^2 =1$ such that $\hat{\omega}(\hat{\alpha}) = \hat{\alpha}(\hat{\omega}_{a*}+ \hat{\omega}_{b*}\bi + \hat{\omega}_{c*}\bj + \hat{\omega}_{d*}\bk )$ is also a solution of \eqref{qua_sys_abcd3} with $ |\hat{\omega}(\hat{\alpha})|=\hat{\alpha}$ for any $\hat{\alpha} > 0$.
    To extend the validness of \eqref{vAv_k} to matrix $A_4$ for some $\hat{\alpha}$, i.e.,
    \begin{equation}\label{vAv_4_opt}
    (1 + \hat{\alpha}^2)\bx_1^HA_4\bx_1= (1 + |\hat{\omega}(\hat{\alpha})|^2)\bx_1^HA_4\bx_1=(1 + \hat{\alpha}^2)\frac{A_4 \bullet X}{r},
    \end{equation}
    we compute
    \begin{align}\label{vAv_4}
    &(1 + |\hat{\omega}(\hat{\alpha})|^2)\bx_1^HA_4\bx_1 \\ \notag
        = &(\overline{\hat{\omega}(\hat{\alpha})}\bv_1^H + \bv_2^H)A_4(\bv_1\hat{\omega}(\hat{\alpha}) + \bv_2) \\ \notag
        = &|\hat{\omega}(\hat{\alpha})|^2\bv_1^HA_4\bv_1 + \bv_2^HA_4\bv_2 + 2\re(\overline{\hat{\omega}(\hat{\alpha})}\bv_1^HA_4\bv_2)\\ \notag
        = &\hat{\alpha}^2\bv_1^HA_4\bv_1 + \bv_2^HA_4\bv_2 + 2\re(\overline{\hat{\omega}(\hat{\alpha})}\bv_1^HA_4\bv_2).
    \end{align}
    Denoting $\tau_4:=\bv_1^HA_4\bv_2 = \hat{a}_4 + \hat{b}_4\bi + \hat{c}_4\bj + \hat{d}_4\bk$, triangular representations of $\tau_4$ and $\hat{\omega}(\hat{\alpha}) $ gives that
    $$
    \hat{\alpha}\hat{\iota}:= \hat{\alpha}|\tau_4|\left(\cos\hat{\theta}_{\hat{\alpha}}\cos\hat{\theta}_4 - \sin\hat{\theta}_{\hat{\alpha}}\sin\hat{\theta}_4 \re({\hat{\omega}^I(\hat{\alpha})} \cdot \tau_4^I)\right)=\re(\overline{\hat{\omega}(\hat{\alpha})}\bv_1^HA_4\bv_2),
    $$
   where  $\hat{\omega}(\hat{\alpha}) = |\hat{\omega}(\hat{\alpha})|(\cos\hat{\theta}_{\hat{\alpha}} + {\hat{\omega}^I(\hat{\alpha})}\sin\hat{\theta}_{\hat{\alpha}})$ and $\bv_1^HA_4\bv_2 = |\tau_4|(\cos\hat{\theta}_4 + \tau_4^I\sin\hat{\theta}_4)$.
    Combining the above equality with \eqref{vAv_4}, identity \eqref{vAv_4_opt} that we want to prove is equivalent to the following real quadratic equation in terms of $\hat{\alpha}$:
    \begin{equation}\label{root_function2}
        \left(\bv_1^HA_4\bv_1 - \frac{A_4 \bullet X}{r}\right)\hat{\alpha}^2 + 2\hat{\iota}\hat{\alpha} + \left(\bv_2^HA_4\bv_2 - \frac{A_4 \bullet X}{r}\right) = 0,
    \end{equation}
    where $\hat{\iota}$ is independent of $\hat \alpha$ as $\sin\hat{\theta}_{\hat{\alpha}}$, $\cos\hat{\theta}_{\hat{\alpha}}$ and $\hat{\omega}^I(\hat{\alpha})$ are all independent of $\hat \alpha$.
    Note that \eqref{root_function2} must have two real roots with opposite signs since $\bv_1^HA_4\bv_1 > \frac{A_4 \bullet X}{r}$ and $\bv_2^HA_4\bv_2 < \frac{A_4 \bullet X}{r}$. Let $\hat{\alpha}_*$ be the positive root of \eqref{root_function2} and construct $\hat{\omega}_* = \hat{\alpha}_*(\hat{\omega}_{a*} + \hat{\omega}_{b*}\bi + \hat{\omega}_{c*}\bj + \hat{\omega}_{d*}\bk)$. Then \eqref{vAv_4_opt} holds with $\hat{\alpha} = \hat{\alpha}_*$, and thus $\bv_1^HA_k\bv_1 = \frac{A_k \bullet X}{r}$ for $k = 1, 2, 3, 4$.
    Recall that $X = \sum_{i=1}^r\bx_i\bx_i^H$, we have
    $$ X - \bx_1\bx_1^H \overset{\eqref{x-equal-v}}= \sum_{j = 3}^r\bv_j\bv_j^H + \bx_2\bx_2^H \succeq 0. $$
    Therefore, rank$(X - \bx_1\bx_1^H) = r -1$, $ \bv_i^HA_k\bv_i \overset{\eqref{decom_X_123}} = \frac{A_k \bullet X}{r}$, $i = 3,\cdots,r$ for $k = 1, 2, 3$, and \begin{eqnarray*}\bx_2^HA_k\bx_2 &=&A_k \bullet (\bx_2 \bx_2^H) = A_k \bullet \left(X - \sum_{j = 3}^r\bv_j\bv_j^H - \bx_1\bx_1^H\right) \\
    	& \overset{\eqref{vAv_k}\,\eqref{decom_X_123}}=&A_k \bullet X - \frac{A_k \bullet X }{r} - (r-2) \frac{A_k \bullet X }{r} = \frac{A_k \bullet X }{r},~~\forall ~~k=1,2,3.
    	\end{eqnarray*}
    Repeating the above process recursively on $X - \bx_1\bx_1^H$, we will finally obtain a rank-one decomposition of $X$ such that \eqref{decom_X} holds for $k=1,2,3,4$.
\end{proof}

Note that the rank-one decomposition procedure we provide in the proof of Theorem \ref{qua_decomposition} is actually implementable. To make it clear, we summarize all the computational procedures in Algorithm \ref{ACQMD1} below.
In particular, Step $1$ and Step $2$ in the proof are described by Algorithm \ref{ACQMD1} with $\ell =3$ and $\ell =4$, respectively.

\begin{breakablealgorithm}
    \caption{Algorithm for computing the quaternion rank-one decomposition of $\ell$ matrices (ACQRD)}
    \label{ACQMD1}
    \begin{algorithmic}
        \STATE{$\boldsymbol{Input:}$ Hermitian matrices $A_k\in\mcH^n$ for $k=1,\cdots,\ell$, a rank-$r$ Hermitian \\
               \qquad positive semidefinite matrix $X\in\mcH_+^n$ with $X = \sum_{i=1}^r\bu_i\bu_i^H$ such that \\
               \qquad $\bu_i^HA_k\bu_i = \frac{A_k \bullet X}{r}, ~\mbox{\rm for}~ i = 1,\cdots,r~\mbox{\rm and}~ k = 1, \cdots, \ell-1.$
                }
        \STATE{$\boldsymbol{Output:}$ $X = \sum_{i=1}^r\bx_i\bx_i^H$, such that
         $
               \bx_i^HA_k\bx_i = \frac{A_k \bullet X}{r}, \mbox{\rm for } i = 1,\cdots,r$ \\
              \qquad and $k = 1, \cdots, \ell.$}
        \STATE{\qquad Let $\bU=\{\bu_1, \cdots, \bu_r  \}$ and $i=1$. \\
        \qquad\textbf{for $j=1,\cdots, r$}}
        \STATE{\qquad \qquad \textbf{If} $\bu_j^HA_{\ell}\bu_j= \frac{A_{\ell} \bullet X}{r}$, then $\bx_i = \bu_j$, $i=i+1$, and $\bU = \bU/\bu_j $.}
        \STATE{\qquad \textbf{end} \\
        	
        \qquad\textbf{repeat}}
        \STATE{\qquad\quad 1. Find $\hat \bu_1$ and $\hat \bu_2 \in \bU$ such that $\hat \bu_1^HA_{\ell}\hat \bu_1 - \frac{A_{\ell} \bullet X}{r} > 0$ and $\hat \bu_2^HA_{\ell}\hat \bu_2 -$ \\
        \qquad\quad $\frac{A_{\ell} \bullet X}{r} < 0$. Let $\hat \bu_1^HA_k \hat \bu_{2} = a_k + b_k\bi + c_k\bj + d_k\bk$ for $k=1,\cdots,\ell-1$ \\
        \qquad\quad and compute one solution $(\omega_{a*},\omega_{b*},\omega_{c*},\omega_{d*})$ with $\omega_{a*}^2 + \omega_{b*}^2 + \omega_{c*}^2+$\\
        \qquad\quad $ \omega_{d*}^2 =1$ of the following the {equations}
        	 \begin{equation*}
        	a_k\omega_a + b_k\omega_b + c_k\omega_c + d_k\omega_d = 0~~\mbox{for}~~k=1,\cdots, \ell-1.
        	\end{equation*}
       }	
        \STATE{\qquad\quad 2. Compute the positive root $\alpha^*$ of the real quadratic equation:
        \begin{equation*}
        \small\qquad\quad\left(\hat \bu_1^HA_{\ell}\hat \bu_1 - \frac{A_{\ell} \bullet X}{r}\right)\alpha^2 + 2\iota\cdot\alpha + \left(\hat \bu_{2}^HA_{\ell}\hat\bu_{2} - \frac{A_{\ell} \bullet X}{r}\right) = 0,
        \end{equation*}
        \qquad\quad where $\small\iota = |\tau_{\ell}|\left(\cos\theta_\alpha\cos\theta_{\ell} - \sin\theta_\alpha\sin\theta_{\ell}\mbox{\rm Re}({\omega^I(\alpha)} \cdot {\tau^I_{\ell}})\right)$, $\small (\cos\theta_\alpha, \sin\theta_\alpha,$\\
        \qquad\quad ${\omega^I(\alpha)} )$ and $(\cos\theta_{\ell}, \sin\theta_{\ell}, \tau^I_{\ell})$ are the triangular-representation-triple \\
        \qquad\quad of $\omega(\alpha)=\alpha(\omega_{a*}+\omega_{b*}\bi+\omega_{c*}\bj+\omega_{d*}\bk)$ and $\tau_{\ell}=\hat \bu_1^HA_{\ell}\hat\bu_{2}$, respectively.}
        \STATE{\qquad\quad 3. Set $\omega_* = \alpha_* (\omega_{a*}+\omega_{b*}\bi+\omega_{c*}\bj+\omega_{d*}\bk)$,
        $$\hat \bv_1 = \frac{\hat \bu_1\omega_* + \bu_{2}}{\sqrt{1 + |\omega_*|^2}} \ \mbox{and} \ \hat \bv_{2} = \frac{-\hat\bu_1 + \hat\bu_{2}\overline{\omega_*}}{\sqrt{1 + |\omega_*|^2}}.$$}
        \STATE{\qquad\quad 4. Let $\bx_i := \hat \bv_1$, $i=i+1$, $\bU = \bU / \hat \bu_1$. \textbf{If} $\hat \bv_2^{H} A_{\ell} \hat \bv_2  = \frac{A_{\ell}\bullet X}{r}$, then \\
        \qquad\quad $\bx_i := \hat \bv_2$, $i=i+1$, $\bU = \bU / \hat \bu_2$. \textbf{Else} $\bU = (\bU/\hat \bu_2) \bigcup \hat \bv_2$.}
        \STATE{\qquad \textbf{until} $i=r+1$.}
    \end{algorithmic}
\end{breakablealgorithm}

\section{The Joint Numerical Range}\label{sec:joint_num_range}
Numerical range is an important tool in linear algebra and convex analysis, and it has wide  applications in spectrum analysis \cite{rasulov2019description}, quantum computing \cite{dirr2006new,rodman2016continuity}, engineering, etc. Joint numerical range was first proposed in 1979 and is an extension of numerical range \cite{au1979remark}, which mainly
focuses on the geometric properties like convexity of the joint field values of several matrices. Joint numerical range also has as wide applications as that of numerical range, and theoretically, it has a close connection to other fundamental results such as $\mathcal{S}$-Lemma, which will be discussed in the next section. On the other hand, people generalized the results of joint numerical ranges in various ways \cite{szymanski2018classification,li2009joint,brickman1961field,ai2011new}. In this section, we show how to further extend the classical results on the convexity of the joint numerical ranges \cite{au1979remark,pang2004joint} to the quaternion domain via the rank-one decomposition of quaternion matrices studied in the last section. We start with some basic concepts that will be used later.
\begin{definition}
    Let $A$ be any $n \times n$ quaternion matrix, the $field \ of \ values$ of A is given by
    \begin{equation*}
        \mathcal{F}(A) := \left\{ \bx^HA\bx \ : \ {\bx^H\bx = 1}, \bx \in \mH^n \right\} \subseteq \mathbb{H}.
    \end{equation*}
\end{definition}
With this concept in hand, we formally define the joint numerical range of a set of matrices as follows.
\begin{definition}
    The $joint \ numerical \ range$ of $n \times n$ quaternion matrices $A_1, \cdots, A_m$ is defined to be
    \begin{equation*}
        \mathcal{F}(A_1, \cdots, A_m) := \left\{ \begin{pmatrix} \bx^HA_1\bx \\ \bx^HA_2\bx \\ \vdots \\ \bx^HA_m\bx \end{pmatrix} ~: ~{\bx^H\bx = 1}, \bx \in \mH^n \right\}.
    \end{equation*}
\end{definition}
Regarding the convexity of the joint numerical ranges, Hausdorff \cite{hausdorff1919wertvorrat} showed that
\begin{center}`` If $A_1$ and $A_2$ are complex Hermitian, then $\mathcal{F}(A_1, A_2)$ is a convex set."\end{center}
In a slightly different form, Brickman \cite{brickman1961field} extended the above result to three matrices, i.e., suppose $A_1, A_2, A_3$ are $n \times n$ complex Hermitian matrices, then
\begin{equation}\label{form-Brickman}\left\{ \begin{pmatrix} \bx^HA_1\bx \\ \bx^HA_2\bx \\ \bx^HA_3\bx \end{pmatrix} ~:~ \bx \in \mC^n \right\}
\end{equation}
is a convex set, where $\mC^n$ is the set of complex vectors.

As a matter of fact, the convexity of the joint numerical ranges in complex domain can also be extended to quaternion domain for more matrices. In particular, Au-Yeung and Poon \cite{au1979remark}  established the following result for quaternion Hermitian matrices.

\begin{theorem}[Au-Yeung and Poon \cite{au1979remark}]
\label{Yeung_theo}
    If $n \not= 2$ and $A_1, \cdots, A_5 \in \mcH^{n}$ are quaternion hermitian matrices, then
    \begin{equation}
        \left\{ \begin{pmatrix} \bx^HA_1\bx \\ \bx^HA_2\bx \\ \vdots \\ \bx^HA_5\bx \end{pmatrix} ~:~ \bx^H\bx = 1, \bx \in \mH^n \right\}
    \end{equation}
    is a convex set.
\end{theorem}
Note that the above theorem assumes $n \not= 2$ and this condition is necessary as
a counterexample is presented in \cite{rodman2014topics}. However, such condition can be relaxed if we consider the slightly different form of \eqref{form-Brickman} by Brickman, where the restriction $\bx^H\bx=1$ is dropped. We present this result in Theorem \ref{joint_num_range_quaternion}, which is proved by our quaternion matrix rank-one decomposition theorem.

\begin{theorem}\label{joint_num_range_quaternion}
    Suppose that $A_k \in \mcH^{n}$, $k = 1, \cdots, 5$. Then
    \begin{equation}
        \left\{ \begin{pmatrix} \bx^HA_1\bx \\ \bx^HA_2\bx \\ \cdots \\ \bx^HA_5\bx \end{pmatrix} ~:~ \bx \in \mH^n \right\}
    \end{equation}
    is a convex set.
\end{theorem}
\begin{proof}
    To complete the proof, it suffices to show
    \begin{equation}
        \left\{ \begin{pmatrix} \bx^HA_1\bx \\ \bx^HA_2\bx \\ \cdots \\ \bx^HA_5\bx \end{pmatrix} ~:~ \bx \in \mH^n \right\} = \left\{ \begin{pmatrix} A_1 \bullet X \\ A_2 \bullet X \\ \cdots \\ A_5 \bullet X \end{pmatrix} ~:~ X \succeq 0 \right\} \notag
    \end{equation}
    as the right-hand side is a convex set.
    Noting that the right-hand side is the convex hull of the left-hand side, we only need to show the right-hand side is included by the left-hand side.

    Take any nonzero vector
    \begin{equation}
        \begin{pmatrix} v_1 \\ v_2 \\ \cdots \\ v_5 \end{pmatrix} \in \left\{ \begin{pmatrix} A_1 \bullet X \\ A_2 \bullet X \\ \cdots \\ A_5 \bullet X \end{pmatrix} ~:~ X \succeq 0 \right\}. \notag
    \end{equation}
    Without loss of generality, suppose $v_5 \not= 0$ and $(A_k - \frac{v_k}{v_5}A_5) \bullet X = 0, ~{\rm for} \ k = 1, 2, 3, 4$.
    Then, by Theorem \ref{qua_decomposition}, there exits a rank-one decomposition $X = \sum_{i=1}^r\bx_i\bx_i^H$ such that
    \begin{equation}\label{sys_v1_v5}
        \bx_i^H\left(A_k - \frac{v_k}{v_5}A_5\right)\bx_i = \frac{(A_k - \frac{v_k}{v_5}A_5) \bullet X}{r} = 0, \ k = 1,2, 3, 4
    \end{equation}
    for $i = 1, \cdots, r$.
    Since $v_5 = A_5 \bullet X = \sum_{i=1}^r\bx_i^HA_5\bx_i$, there is at least one vector, say $\bx_1$ such that $\bx_1^HA_5\bx_1$ shares the same sign as $v_5$.
    Let $\rho = \sqrt{\frac{v_5}{\bx_1^HA_5\bx_1}}$ and $\bx = \rho \bx_1$, then we have
    \begin{equation*}
        \bx^HA_5\bx = \rho^2\bx_1^HA_5\bx_1 = v_5,
    \end{equation*}
    and \eqref{sys_v1_v5} also holds for $\bx$.
    Plugging $\bx$ and $\bx^HA_5\bx = v_5$ into \eqref{sys_v1_v5}, we have
    \begin{equation*}
    \bx^HA_k\bx=\bx^H\frac{v_k}{v_5}A_5\bx=\frac{v_k}{v_5}\cdot v_5=v_k
    \end{equation*}
    for $k = 1, 2, 3, 4$. Hence the vector $ \begin{pmatrix} v_1 \\ v_2 \\ \cdots \\ v_5 \end{pmatrix} \in \left\{ \begin{pmatrix} \bx^HA_1\bx \\ \bx^HA_2\bx \\ \cdots \\ \bx^HA_5\bx \end{pmatrix} ~:~ x \in \mH^n \right\}$ and   the conclusion follows.
\end{proof}

\section{S-Procedure}\label{sec:s_lem}
It is well known that $\mathcal{S}$-procedure plays an important role in robust optimization \cite{anitescu2000degenerate,anitescu2002superlinearly,goldfarb2003robust}, statistics \cite{hoerl1970ridge}, signal processing \cite{luo2003applications}, control and stability problems \cite{ben2001lectures,boyd1994linear}, among others. In this section, we discuss
another interesting {theoretical implication} of our new rank-one decomposition (Theorem \ref{qua_decomposition}) in $\mathcal{S}$-procedure. We first recall the following lemma, which is due to \cite{yakubovich1971s}, about
$\mathcal{S}$-procedure in real domain.

\begin{lemma}{(\bf $\mathcal{S}$-Procedure)} Let $F(\bx)=\bx^{\top}A_0\bx + 2\bb_0^{\top}\bx + c_0$ and $G_i(\bx) = \bx^{\top}A_i\bx + 2\bb_i^{\top}\bx + c_i$, $i=1,\ldots,m$ be quadratic functions of $\bx \in \mR^n$. Then $F(\bx) \ge 0$ for all $\bx$ such that $G_i(\bx)\ge 0$, $i=1,\ldots,m$, if there exist $\tau_i \ge 0 $ such that
\begin{eqnarray*}
 \left[ \begin{array}{cc}c_0 & \bb_0^{\top}\\
 \bb_0 & A_{0}\\
\end{array} \right]
- \sum\limits_{i=1}^{p}\tau_i \cdot
\left[ \begin{array}{cc}c_i & \bb_i^{\top}\\
 \bb_i & A_i
 \end{array} \right] \succeq 0.
 \end{eqnarray*}
Moreover, if $m=1$ then the converse holds if there exists $\bx_0$ such that $F(\bx_0)>0$.
\end{lemma}
We remark that the above result essentially study the relationship between
\begin{equation}\label{S-procedure1}
G_1(\bx) \ge 0, G_2(\bx) \ge 0, \cdots, G_m(\bx) \ge 0 \Rightarrow F(\bx) \ge 0,
\vspace{-1.5mm}
\end{equation}
and
\begin{equation}\label{S-procedure2}
\exists \tau_1 \ge 0, \tau_2 \ge 0, \cdots, \tau_m \ge 0 \ {\rm such \ that} \ F(\bx) - \sum_{i=1}^m \tau_iG_i(\bx) \ge 0 \ \forall \bx.
\end{equation}
It is obvious that \eqref{S-procedure2} implies \eqref{S-procedure1}, and the converse also holds when $m=1$.
In the following, we call $\mathcal{S}$-procedure is {\it lossless} if \eqref{S-procedure1} and \eqref{S-procedure2} are equivalent. Although $\mathcal{S}$-procedure was first studied in the real domain, interestingly, a stronger result has been established in the complex domain.

\begin{lemma}{(Yakubovich \cite{yakubovich1971s})}
    Suppose $F, G_1, G_2$ are Hermitian forms and satisfy \eqref{S-procedure1}. Moreover, there is $\bx_0 \in \mathbb{C}^n$ such that $G_i(\bx_0) > 0, i = 1,2$. Then the S-procedure is lossless for $m = 2$.
\end{lemma}

Inspired by such result, we further consider $\mathcal{S}$-procedure in the quaternion domain and manage to show that it is {\it lossless} for $m = 4$ by Theorem \ref{qua_decomposition}.

\begin{theorem}
    Suppose $F, G_1, G_2, G_3, G_4$ are quaternion Hermitian forms and satisfy \eqref{S-procedure1}. Moreover, there is \begin{equation}\label{s-procedure-x0}
   \bx_0 \in \mathbb{H}^n~\mbox{such that}~G_i(\bx_0) > 0,~ i = 1, 2, 3, 4.
    \end{equation}
    Then the S-procedure is lossless for $m = 4$.
\end{theorem}
\begin{proof} To prove the conclusion, it suffices to establish \eqref{S-procedure2} for $m=4$.
    Let $F(\bx) = \bx^HA_0\bx$ and $G_i(\bx) = \bx^HA_i\bx$ for $i = 1, 2, 3, 4$. Consider the following set
    $$
    \mathbb{D}=\left\{ \begin{pmatrix} \bx^HA_0\bx \\ \bx^HA_1\bx \\ \cdots \\ \bx^HA_4\bx \end{pmatrix} ~:~ \bx \in \mH^n \right\},
    $$
    which is a convex cone in $\mR^5$ by Theorem \ref{joint_num_range_quaternion}. Moreover, since $F(\bx)$ and $G_i(\bx)$, $i = 1, \cdots, 4$ satisfy \eqref{S-procedure1}, i.e.,
    \begin{equation*}
        G_1(\bx) \ge 0, \cdots, G_4(\bx) \ge 0 \Rightarrow F(\bx) \ge 0 ,
    \end{equation*}
    we have $\mathbb{D}\cap\mathbb{E}=\emptyset$, where $\mathbb{E} = \{(y_0,y_1,y_2,y_3,y_4) |~ y_0<0,~ y_i>0,~i=1,2,3,4\}$. Noting that $\mathbb{E} $ is also a convex cone, the standard separating hyperplane theorem \cite{boyd2004convex} implies that there exists $\bt=(t_0, t_1, t_2, t_3, t_4) \not= 0$ such that
    \begin{equation}\label{separetion_ineq1}
        \sum_{k = 0}^4t_ky_k \le 0 ~{\rm for~ all}~ (y_0, y_1, y_2, y_3, y_4)\in\mathbb{E}
    \end{equation}
    and
    \begin{equation}\label{separetion_ineq2}
        \sum_{k = 0}^4t_k\bx^HA_k\bx \ge 0 ~{\rm for~ all}~ \bx \in \mH^n.
    \end{equation}
    Then we must have $t_0 \ge 0$ and $t_k \le 0$ for $k = 1, 2, 3, 4$ by inequality \eqref{separetion_ineq1}. Moreover, we can confirm that $t_0 > 0$, otherwise since $\bt \neq 0 $ there exists some $k \in \{1,2,3,4\}$ such that $t_k < 0$ and thus for $\bx_0$ in \eqref{s-procedure-x0} it holds that $\sum_{k = 0}^4t_k\bx_0^HA_k\bx_0 = \sum_{k = 1}^4t_k\bx_0^HA_k\bx_0 < 0$, which  contradicts with \eqref{separetion_ineq2}. Finally, the validness of \eqref{S-procedure2} with $m=4$ is implied by \eqref{separetion_ineq2}  by
    dividing both sides of \eqref{separetion_ineq2} by $t_0$ and letting $\tau_k = \frac{-t_k}{t_0} \ge 0  $ for $k=1,2,3,4$.
\end{proof}

\section{Quaternion Quadratically Constrained Quadratic Optimization}\label{sec:qua_quadratic_opt}
In this section, we consider the following {quadratically constrained quadratic programming} in the quaternion domain:
\begin{equation*}\label{QCQP_stand}
    \begin{aligned}
        {\rm (QCQP)} \quad  \max \quad & \bx^HQ\bx + 2{\rm Re}(\bx^Hq)\\
                      {\rm s.t.} \quad & \bx^HA_j\bx + 2{\rm Re}(\bx^Hb_j) + c_j \le 0, \ j = 1,\cdots,m,
    \end{aligned}
\end{equation*}
where $Q, A_j \in \mcH^n$, $b_j, q \in \mH^n$, $c_j \in \mR$, $j = 1,\cdots,m$. Normally, solving the above problem is very challenging. However, when the number of constraints $m$ is small, the problem is possibly tractable. For instance, in the complex domain, Huang and Zhang \cite{huang2007complex} showed that the problem (QCQP) with $m=2$ could be cast as an SDP and thus could be solved in polynomial time. While regarding the quaternion domain, we shall show that a similar result holds for a larger value of $m$ (in particular for $m=4$) by the afore-mentioned rank-one decomposition in Theorem \ref{qua_decomposition}.

To relate this problem to SDP, we rewrite (QCQP) as the following matrix form:
\begin{equation}\label{QCQP2}
\begin{aligned}
 \quad  \max \quad & B_0 \bullet \begin{bmatrix}
1 & \bx^H \\
\bx & \bx\bx^H \\
\end{bmatrix}\\
{\rm s.t.} \quad & B_j \bullet \begin{bmatrix}
1 & \bx^H \\
\bx & \bx\bx^H \\
\end{bmatrix} \le 0, ~j = 1,\cdots,m,
\end{aligned}
\end{equation}
where $B_0 = \begin{bmatrix}
0 & q^H \\
q & Q \\
\end{bmatrix}, \
B_j = \begin{bmatrix}
c_j & b_j^H \\
b_j & A_j \\
\end{bmatrix}, ~j = 1,\cdots,m$. Problem \eqref{QCQP2} can be homogenized by introducing a new variable $t$ and requiring $|t|^2=1$:
\begin{equation*}\label{HQ}
\begin{aligned}
{\rm (HQCQP)} \quad  \max \quad & B_0 \bullet \begin{bmatrix}
|t|^2 & t \cdot \bx^H \\
\bx \cdot \bar{t}  & \bx\bx^H \\
\end{bmatrix}\\
{\rm s.t.} \quad & B_j \bullet \begin{bmatrix}
|t|^2 & t \cdot \bx^H \\
\bx \cdot \bar{t}  & \bx\bx^H \\
\end{bmatrix} \le 0, ~j = 1,\cdots,m, \\
& B_{m+1} \bullet \begin{bmatrix}
|t|^2 & t \cdot \bx^H \\
\bx \cdot \bar{t}  & \bx\bx^H \\
\end{bmatrix} = 1,
\end{aligned}
\end{equation*}
where $ B_{m+1} = \begin{bmatrix}
1 & 0 \\
0 & 0 \\
\end{bmatrix}$. Then for any solution $\begin{pmatrix} t \\ \bx \\ \end{pmatrix}$  of (HQCQP), since $|t|^2=1$, $\bx \cdot \bar {t}$ is a solution to \eqref{QCQP2} and hence a solution to (QCQP). Therefore, we shall work on (HQCQP) in the rest of this section. By letting
$$X = \begin{bmatrix}
|t|^2 & t \cdot \bx^H \\
\bx \cdot \bar{t}  & \bx\bx^H \\
\end{bmatrix}~\mbox{and dropping the constraint:}~\rank (X) =1,$$
we obtain an SDP relaxation of (HQCQP):
\begin{equation*}\label{QCQPR}
    \begin{aligned}
        {\rm (QCQPR)} \quad  \max \quad & B_0 \bullet X \\
                       {\rm s.t.} \quad & B_j \bullet X \le 0, ~j = 1,\cdots,m, \\
                                        & B_{m+1} \bullet X = 1, \\
                                        & X \succeq 0,
    \end{aligned}
\end{equation*}
whose dual is given by
\begin{equation*}\label{DQCQPR}
    \begin{aligned}
        {\rm (DQCQPR)} \quad  \min \quad & y_0 \\
                        {\rm s.t.} \quad & Y = \sum_{j=1}^m y_jB_j - B_0 + y_0B_{m+1} \succeq 0, \\
                                         & y_j \ge 0,~j = 1,\cdots,m,~y_0 \ {\rm free}.
    \end{aligned}
\end{equation*}
To proceed, we assume (QCQP) satisfies the Slater condition, that is, there exists $\bx_0 \in \mH^n$ such that \begin{equation}\label{slater}
q_j(\bx_0) := B_j \bullet \begin{bmatrix}
    1 & \bx_0^H \\
    \bx_0 & \bx_0\bx_0^H \\
\end{bmatrix} < 0, j = 1,\cdots,m.
\end{equation}
Accordingly, (QCQPR) satisfies the Slater condition as well. Now we are ready to present the main theorem of this section, which states that (QCQP) is essentially equivalent to an SDP when $m = 4$.

\begin{theorem}
    Suppose (QCQP) satisfies the Slater condition \eqref{slater} and $m = 4$. Then (QCQP) and (QCQPR) have the same optimal value. Moreover, an optimal solution to (QCQP) can be constructed from that of (QCQPR).
\end{theorem}
\begin{proof}
    Denote `$\trianglelefteq$' to be either `$<$' or `='. For the primal optimal solution $X^*$,  the first group of constraints in (QCQPR) can be rewritten as
    $$B_j \bullet X^* \trianglelefteq_j 0, ~j = 1,2,3,4.$$
    Letting $r=\rank X^*$,
     by Theorem \ref{qua_decomposition}, there exist some non-zero quaternion vectors $\tilde \bx_k = \begin{pmatrix}
     t_k \\
     {\bx}_k \\
     \end{pmatrix} \in \mH^{n+1}, ~k = 1,\cdots,r$, such that
    \begin{equation}\label{QCQP-Decom}
        X^* = \sum_{k=1}^r \tilde\bx_k\tilde\bx_k^H \ {\rm ~and~} \ B_j \bullet \tilde\bx_k\tilde\bx_k^H \trianglelefteq_j 0,~\mbox{for}~ \ j = 1,2,3,4, \ k = 1,\cdots,r.
    \end{equation}
    Since $B_{m+1} \bullet X^* = 1$, we have
   $\sum_{k=1}^{r}|t_k|^2 = X^*_{11} = 1$ and there exits $\ell \in \left\{ 1, \cdots, r \right\}$ such that $t_\ell \not= 0$.
    Therefore, according to \eqref{QCQP-Decom}, it holds that
     \begin{eqnarray*}q_j\left({\bx}_\ell \cdot \frac{1}{t_\ell}\right)&=& B_j \bullet \begin{bmatrix}
     1 & \left({\bx}_\ell \cdot \frac{1}{t_\ell}\right)^H \\
     \left({\bx}_\ell \cdot \frac{1}{t_\ell}\right) & \left({\bx}_\ell \cdot \frac{1}{t_\ell}\right)\left({\bx}_\ell \cdot \frac{1}{t_\ell}\right)^H \\
     \end{bmatrix} \\
  &=& \frac{1}{|t_\ell|^2} \cdot B_j \bullet \begin{bmatrix}
  	|t_\ell|^2 & {t_\ell} \cdot {\bx}_\ell ^H \\
  	{\bx}_\ell \cdot \bar{t_\ell} & {\bx}_\ell {\bx}_{\ell}^H \\
  \end{bmatrix}\\
&=&\frac{1}{|t_\ell|^2} B_j \bullet \tilde\bx_{\ell}\tilde\bx_{\ell}^H \trianglelefteq_j 0, j = 1,2,3,4,
     \end{eqnarray*}
     which further implies $\bx_{\ell} \cdot \frac{1}{t_\ell}$ is a feasible solution to (QCQP).

        Suppose  $(y_0^*, y_1^*, y_2^*, y_3^*, y_4^*, Y^*)$ is a dual optimal solution. Since the Slater condition holds, the strong duality is valid between (QCQPR) and (DQCQPR). Consequently, the complementary slackness condition holds, i.e.,
        $$\sum_{i=1}^{r} Y^* \bullet \tilde\bx_i\tilde\bx_i^H = Y^* \bullet X^*  = 0,~\mbox{and}~y_j^*\cdot(B_j \bullet X^*) = 0 ~{\rm for}~ j = 1,2,3,4.
        $$
        Moreover, since $X^*\succeq 0$ and $Y^*\succeq 0$, we have $Y^* \bullet \tilde \bx_{\ell}\tilde \bx_{\ell}^H = 0$. We also observe from \eqref{QCQP-Decom} that
        $$
        B_j \bullet \tilde\bx_{\ell}\tilde\bx_{\ell}^H \trianglelefteq_j 0 \Leftrightarrow B_j \bullet X^* \trianglelefteq_j 0 ,~\mbox{for}~ \ j = 1,2,3,4,
        $$
        which indicates that
        $$y_j^*\cdot\left(B_j \bullet \tilde\bx_{\ell}\tilde\bx_{\ell}^H\right) = y_j^*\cdot(B_j \bullet X^*) = 0 ~{\rm for}~ j = 1,2,3,4.$$
        Combining the above results with  $\begin{pmatrix}
        1 \\
        {\bx}_{\ell} \cdot \frac{1}{t_{\ell}}\\
        \end{pmatrix}
        \begin{pmatrix}
        1 \\
        {\bx}_{\ell} \cdot \frac{1}{t_{\ell}}\\
        \end{pmatrix}^H = \frac{1}{|t_{\ell}|^2}\cdot \tilde\bx_{\ell}\tilde\bx_{\ell}^H$ yields that $\begin{pmatrix}
        1 \\
        {\bx}_{\ell} \cdot \frac{1}{t_{\ell}}\\
        \end{pmatrix}
        \begin{pmatrix}
        1 \\
        {\bx}_{\ell} \cdot \frac{1}{t_{\ell}}\\
        \end{pmatrix}^H$ is complementary to $(y_0^*, y_1^*, y_2^*, y_3^*, y_4^*, Y^*)$, and it is an optimal solution to (QCQPR). Therefore ${\bx}_{\ell} \cdot \frac{1}{t_{\ell}}$ is optimal for (QCQP) and the conclusion follows.
\end{proof}

%
%



%
%

\bibliographystyle{spmpsci}      

\end{document}